\documentclass[review]{elsarticle}
  \journal{}
   \usepackage{latexsym}
   \usepackage{amssymb}  
 \usepackage{amsfonts,latexsym}
 \usepackage{epsfig}
 \usepackage{multirow}
 \usepackage{amsmath}
 \usepackage{color}
 \usepackage{ulem}
 \usepackage{algorithm,algorithmic}
 \usepackage{epstopdf} 
 \usepackage{graphicx, subfigure}
   \usepackage{algorithm,algorithmic}
   \usepackage{cleveref}

   \usepackage{multirow}
   \newenvironment{proof}{\textbf{Proof}}{$\square$}
    \newtheorem{proposition}{Proposition}[section]
    \newtheorem{theorem}{Theorem}[section]
    \newtheorem{definition}{Definition}[section]
    
   \newtheorem{remark}{Remark}[section]
   
   \makeatletter
   \newcommand{\ssymbol}[1]{^{\@fnsymbol{#1}}}
   \makeatother
   \newcommand{\R}{\mathbb R}
   \newcommand{\C}{\mathbb C}

   \usepackage[mathscr]{euscript}
  
\begin{document}
   	
   	\begin{frontmatter}
   		
   		\title{   Tensor Krylov subspace  methods via the T-product for  large Sylvester  tensor equations }

   		\author[XXX,AAA]{F. BOUYGHF}
   		\ead{fatimabouyghf3@gmail.com}
   		\author[BBB]{M. EL  GUIDE }
   		\ead{mohamed.elguide@um6p.ma}
   		\author[XXX,AAA]{A. EL  ICHI }
   		\ead{elichi.alaa@gmail.com}
   	
   		\address[XXX]{ 
   			LabMIA-SI, University of Mohammed V Rabat, Morocco.\\
   	
   		}
   		\address[BBB]{ 
   			Africa Institute for Research in Economics and Social Sciences (AIRESS), FGSES, Mohammed VI Polytechnic University, Rabat, Morocco.\\
   
   		}
   		\address[AAA]{ 
   		LMPA, University of the Littoral Opal Coast, Calais, France.\\
   		
   	}
   		\cortext[cor1]{Corresponding author}

   	%\maketitle
   	
   	\begin{abstract}
   		In the present paper, we introduce  new tensor krylov subspace    methods  for solving large Sylvester   tensor equations. The proposed method uses the well-known T-product for tensors and tensor subspaces.    We introduce some new tensor products and the related algebraic properties. These new products will enable us to develop   third-order the tensor FOM (tFOM),  GMRES (tGMRES),  tubal Block Arnoldi and   the tensor tubal Block Arnoldi method to solve large Sylvester tensor equation. We give some properties related to these method and  present some numerical experiments.
   	\end{abstract}
   	

   \end{frontmatter}

   	\section{Introduction}
   	
      The  aim of this paper is to present numerical Tensor Krylov subspace methods for solving Sylvester tensor equation 
      \begin{equation}\label{eq1}
      {\mathcal M} (\mathscr{X}) = \mathscr{C},
      \end{equation}
      where  ${\mathcal M}$ is a linear operator  that could be described as 
      \begin{equation}\label{eq2}
      {\mathcal M} (\mathscr{X}) = \mathscr{A} \star\mathscr{X}-\mathscr{X} \star\mathscr{B},
      \end{equation}
        where $\mathscr{A} $, $\mathscr{X}$, $\mathscr{B} $ and $\mathscr{C}$ are  three-way tensors, leaving the specific dimensions to be defined later,  and $\star$ is the T-product  introduced  by  Kilmer and Martin \cite{klimer2,klimer3}. 
      
     Consider the following Sylvester matrix equation 
      \begin{equation}\label{sylvmatrix}
       AX+XB=C.
      \end{equation}
       In the existing body of literature, various methodologies have been proposed to address the solution of Sylvester matrix equations, as articulated in Equation (\ref{sylvmatrix}). When confronted with matrices of relatively small dimensions, established direct methods, as advocated in seminal works such as \cite{Bartels} and \cite{golub2}, are often recommended. These direct methods leverage Schur decomposition to transform the original equation into a more amenable form, thereby facilitating resolution through forward substitution.
       
       For larger Sylvester matrix equations, iterative projection techniques have been advanced, as evidenced by studies like \cite{elguennouni}, \cite{Hu}, and \cite{saad}. These methods employ Galerkin projection approaches, including both classical and block Arnoldi techniques. By employing such projection methods iteratively, lower-dimensional Sylvester matrix equations are derived, subsequently tackled through direct methods for efficient resolution. Notably, comprehensive approaches to Krylov subspace methods for solving linear systems are elucidated in \cite{BMS2}, providing a unified perspective on these iterative methodologies.
       
       It's essential to consider the nature and size of the matrices involved when choosing an appropriate solution strategy, with direct methods favored for smaller matrices and iterative projection methods preferred for handling larger Sylvester matrix equations. The utilization of Schur decomposition and Galerkin projection techniques underscores the adaptability and scalability of these methods across varying problem sizes.
      
    This paper focuses on the development of efficient and robust iterative Krylov subspace methods using the T-product for solving the Sylvester tensor equation (\textit{STE}) represented by (\ref{eq1}). Specifically, when dealing with small-sized tensors in (\ref{eq1}), our aim is to extend the matrix-oriented direct methods outlined in \cite{Bartels} and \cite{golub2} to third-order tensors, employing the T-product formalism. This extension leads to the formulation of the \textit{t-Bartels-Stewart} algorithm.
    
    For larger tensors, we introduce a novel method termed as orthogonal and oblique projection onto a tensor Krylov subspace. Two specific instances of this approach, namely the tensor Full Orthogonalization Method (tFOM) and the tensor Generalized Minimal Residual Method (tGMRES), are examined. Additionally, we present well-known tensor Tubal Block Krylov methods utilizing the T-product to transform the original large Sylvester equation into a lower-dimensional \textit{STE}. In this context, we describe the Tubal Block Arnoldi (TBA) as a generalization of the block Arnoldi matrix.
    
      This paper focuses on the development of efficient and robust iterative Krylov subspace methods using the T-product for solving the Sylvester tensor equation (\textit{STE}) represented by (\ref{eq1}). Specifically, when dealing with small-sized tensors in (\ref{eq1}), our aim is to extend the matrix-oriented direct methods outlined in \cite{Bartels} and \cite{golub2} to third-order tensors, employing the T-product formalism. This extension leads to the formulation of the \textit{t-Bartels-Stewart} algorithm.
      
      For larger tensors, we introduce a novel method termed as orthogonal and oblique projection onto a tensor Krylov subspace. Two specific instances of this approach, namely the tensor Full Orthogonalization Method (tFOM) and the tensor Generalized Minimal Residual Method (tGMRES), are examined. Additionally, we present well-known tensor Tubal Block Krylov methods utilizing the T-product to transform the original large Sylvester equation into a lower-dimensional \textit{STE}. In this context, we describe the Tubal Block Arnoldi (TBA) as a generalization of the block Arnoldi matrix.
      
      This work contributes to the expansion of iterative methods for solving Sylvester tensor equations, catering to both small and large-sized tensors through the adaptation of established matrix-based techniques and the introduction of novel tensor projection approaches.

      User
      reformulate :The paper is organized as follows: In \Cref{sec:section2}, we give notations and definitions related to the T-product.   More details of all notations and definitions given in this section   are given in \cite{1,2,3}.
      In Section 3, we develop  the (tFOM) and (tGMRES) methods.
      The \textit{t-Bartels-Stewart} method will be introduced in Section 4. After defining a tubal  QR factorisation algorithm   in Section 5, we will establish  the tensor tubal-Block Arnoldi process (TBA) that allows us to introduce the tubal-Block Arnoldi for solving large Sylvester tensor equation (TBAS) method.  Finally, some numerical tests are reported in Section 6.
 
      The paper is structured as follows: In Section 2, we present notations and definitions pertaining to the T-product. Comprehensive details for all the notations and definitions introduced in this section can be found in \cite{1,2,3}. Section 3 is dedicated to the development of the tensor Full Orthogonalization Method (tFOM) and the tensor Generalized Minimal Residual Method (tGMRES).  In Section 4, we introduce the \textit{t-Bartels-Stewart} method. Following this, Section 5 outlines the tubal QR factorization algorithm before establishing the Tensor Tubal-Block Arnoldi process (TBA). This process enables the introduction of the Tubal-Block Arnoldi for Solving Large Sylvester Tensor Equation (TBAS) method. Finally, Section 6 presents numerical tests to validate and assess the performance of the proposed methods.

   	\section{Notation and background\label{sec:section2}}

   	   	 	A tensor is  a multidimensional array of data. The number of indices of a tensor is called modes or ways.  	Notice that a scalar can be regarded as a zero mode tensor, first mode tensors are vectors and matrices are second mode tensor. The order of a tensor is the dimensional of the array needed to represent it, also known as
   	ways or modes. 
   	For a given N-mode (or order-N) tensor $ \mathscr {X}\in \mathbb{R}^{n_{1}\times n_{2}\times n_{3}\ldots \times n_{N}}$, the notation $x_{i_{1},\ldots,i_{N}}$ (with $1\leq i_{j}\leq n_{j}$ and $ j=1,\ldots N $) stands for the element $\left(i_{1},\ldots,i_{N} \right) $ of the tensor $\mathscr {X}$. The norm of a tensor $\mathscr{A}\in \mathbb{R}^{n_1\times n_2\times \cdots \times n_\ell}$ is specified by
   	\[
   	\left\| \mathscr{A} \right\|_F^2 = {\sum\limits_{i_1  = 1}^{n_1 } {\sum\limits_{i_2  = 1}^{n_2 } {\cdots\sum\limits_{i_\ell = 1}^{n_\ell} {a_{i_1 i_2 \cdots i_\ell }^2 } } } }^{}.
   	\] Corresponding to a given tensor $ \mathscr {A}\in \mathbb{R}^{n_{1}\times n_{2}\times n_{3}\ldots \times n_{N}}$, the notation $$ \mathop{\mathscr{X}_{\underbrace{::\cdots:}k}}\limits_{\tiny{(N-1)\text{-times}}},\; \; {\rm for }  \quad k=1,2,\ldots,n_{N}$$ denotes a tensor in $\mathbb{R}^{n_{1}\times n_{2}\times n_{3}\ldots \times n_{N-1}}$ which is obtained by fixing the last index and is called frontal slice. Fibers are the higher-order analogue of matrix rows and columns. A fiber is
   	defined by fixing all the indexes  except  one.
   
   	In this paper, a tensor is of third order, i. e., N=3, that will be denoted by the calligraphic script letters, say $\mathscr{A}=[a_{ijk}]_{i,j,k=1}^{n_1,n_2,n_3}$. We use capital letters to denote matrices,
   	lower case letters to denote vectors,   boldface lower case letters to denote tube fibers (tubal
   	scalars or tubes) and boldface upper-case  letters to denote block diagonal matrix. Using MATLAB notation, $\mathscr{A}(:,j,k),\;\mathscr{A}(i,:,k)$ and $\mathscr{A}(i,j,:)$ denote mode-1, mode-2, and
   	mode-3 fibers, respectively.  The  notations $\mathscr{A}(i,:,:),\;\mathscr{A}(:,j,:)$ and $\mathscr{A}(:,:,k)$ stand for  the i-th horizontal,
   	j-th lateral, and k-th frontal slices of $\mathscr{A}$ , respectively. The j-th lateral slice is also denoted by
   	 $\overrightarrow{\mathscr{A}}_j$ .$Tt$ is a
   	 tensor of size ($n_1\times 1 \times n_3 $) and will be referred to as a tensor column. Moreover,  the $k$-th  frontal slices of $\mathscr{A}$ is a matrix  size ($n_1\times   n_2 $)  denoted by $\mathscr{A}^{(k)}$.

   	\subsection{Definitions and properties of the T-product}
   	In this part, we briefly review some concepts and notations related to the T-product, see \cite{braman,klimer3,klimer2} for more details. 	Let $\mathscr {A} \in \mathbb{R}^{n_{1}\times n_{2}\times n_{3}} $ be a third-order tensor, then the operations ${\rm bcirc}$,   unfold and fold are defined by
   	$${\rm bcirc}(\mathscr {A})=\left( {\begin{array}{*{20}{c}}
   		{{\mathscr{A}^{(1)}}}&{{\mathscr{A}^{(n_3)}}}&{{\mathscr{A}^{(n_3-1)}}}& \ldots &{{\mathscr{A}^{(2)}}}\\
   		{\mathscr{A}^{(2)}}&{{\mathscr{A}^{(1)}}}&{{\mathscr{A}^{(n_3)}}}& \ldots &{{\mathscr{A}^{(3)}}}\\
   		\vdots & \ddots & \ddots & \ddots & \vdots \\
   		{{\mathscr{A}^{(n_3)}}}&{{\mathscr{A}^{(n_3-1)}}}& \ddots &{{\mathscr{A}^{(2)}}}&{{\mathscr{A}^{(1)}}}
   		\end{array}} \right)   \in {\R}^{ n_1n_3 \times n_2n_3},$$ 
   	$${\rm unfold}(\mathscr {A} ) = \begin{pmatrix}
   \mathscr{A}^{(1)}  \\
   \mathscr{A}^{(2)}   \\
   	\vdots \\
   	\mathscr{A}^{(n_3)}\end{pmatrix} \in \mathbb{R}^{n_{1}n_{3}\times n_{2}},  \qquad {\rm fold}({\rm unfold}(\mathscr {A}) ) =  \mathscr {A}.$$ 
   	Let $\widetilde {\mathscr {A}}$ be the tensor obtained by applying the  discrete Fourier transform  DFT  matrix  	$F_{n_3} \in {\C}^{n_3\times n_3}$;  (for more details about DFT matrix, see \cite{golub1}) on all the 3-mode tubes of the tensor $\mathscr {A}$. With the Matlab command ${\tt fft}$, we have
   	$$\widetilde {\mathscr {A}}= {\tt fft}(\mathscr {A},[ ],3), \; {\rm and }\;\; {\tt ifft} (\widetilde {\mathscr {A}}, [ ],3)= \mathscr {A},$$
   	where ${\tt ifft}$ denotes the Inverse Fast Fourier Transform.\\
  \noindent 	Let ${\bf A}$ be the matrix 
   	\begin{equation}\label{dft9}
   	{\bf A}= {\rm Diag}(\widetilde {\mathscr {A}})= \left (
   	\begin{array}{cccc}
   	\widetilde {\mathscr {A}}^{(1)}& &&\\
   	& \widetilde {\mathscr {A}}^{(2)}&&\\
   	&&\ddots&\\
   	&&&\widetilde {\mathscr {A}}^{(n_3)}\\
   	\end{array}
   	\right),
   	\end{equation}
   	and the matrices $\widetilde {\mathscr {A}}^{(i)}$'s are the frontal slices of the tensor $\widetilde {\mathscr {A}}$.
   	The block circulant matrix ${\rm bcirc}(\mathscr {A})$ can   be block diagonalized by using the  DFT matrix  and this gives
   	\begin{equation}\label{dft8}
   	(F_{n_3} \otimes I_{n_1})\, {\rm bcirc}(\mathscr {A})\, 	(F_{n_3}^{*} \otimes I_{n_2})={\bf A} 
   	\end{equation}	
   	where  	$F_{n_3}^{*}$ denotes the conjugate transpose of $ F_{n_3}$ and $\otimes$ is the Kronecker matrix product.
   	  Next we recall the definition of the T-product.	
   	
   	\medskip
   	
   	\begin{definition}
   		The \textbf{T-product} ($\star $) between  two tensors
   		$\mathscr {A} \in \mathbb{R}^{n_{1}\times n_{2}\times n_{3}} $ and $\mathscr {B} \in \mathbb{R}^{n_{2}\times m\times n_{3}} $ is the  ${n_{1}\times m\times n_{3}}$ tensor  given by:	
   		$$\mathscr {A} \star \mathscr {B}={\rm fold}({\rm bcirc}(\mathscr {A}){\rm unfold}(\mathscr {B}) ).$$
   	\end{definition}
   	Notice that from the relation \eqref{dft9}, we can show that the   product $\mathscr {C}=\mathscr {A} \star \mathscr {B}$ is equivalent to ${\bf C}= {\bf A}  {\bf B}$. So, the efficient way to compute the T-product is to use Fast Fourier Transform (FFT). 
   	  The following algorithm allows us to compute in an efficient way the T-product of the tensors $\mathscr {A}$ and 
   	$\mathscr {B}$.
   	
   	\begin{algorithm}
   		\caption{Computing the  T-product via FFT}\label{algo1}
   		Inputs: $\mathscr {A} \in \mathbb{R}^{n_{1}\times n_{2}\times n_{3}} $ and $\mathscr {B} \in \mathbb{R}^{n_{2}\times m\times n_{3}} $\\
   		Output: $ \mathscr {C}= \mathscr {A} \star \mathscr {B}  \in \mathbb{R}^{n_{1}\times m \times n_{3}}.$
   		\begin{enumerate}
   			\item Compute $\mathscr {\widetilde A}={\tt fft}(\mathscr {A},[ ],3)$ and $\mathscr {\widetilde B}={\tt fft}(\mathscr {B},[ ],3)$.
   			\item Compute each frontal slices of $\mathscr {\widetilde C}$ by
   			$$C^{(i)}= \left \{
   			\begin{array}{ll}
   			A^{(i)} B^{(i)} , \quad \quad\quad  i=1,\ldots,\lfloor \displaystyle \frac{{n_3}+1}{2} \rfloor\\
   			conj({C}^{(n_3-i+2)}),\quad \quad i=\lfloor \displaystyle \frac{{n_3}+1}{2} \rfloor+1,\ldots,n_3 .
   			\end{array}
   			\right.$$
   			\item Compute $\mathscr {C}={\tt ifft}(\widetilde {C},[],3)$.	 	
   		\end{enumerate}
   	\end{algorithm}

   	\noindent For the T-product, we have the following definitions

   	\begin{definition}\cite{klimer2}\label{identens}
   		\begin{enumerate}
   			\item The identity tensor $\mathscr{I}_{n_{1}n_{1}n_{3}} $ is the tensor whose first frontal slice is the identity matrix $I_{n_1n_1}$ and the other frontal slices are all zeros.
   			\item Let ${\rm \bf z}\in {\mathbb R}^{1\times 1 \times n_{3}} $,
   			the  tubal rank of  ${\rm \bf z}$ is the number of its
   			non-zero Fourier coefficients. If the tubal-rank of ${\rm \bf z}$ is equal to  $n_3$, we say that  it is invertible and   we denote by $({\rm \bf z})^{-1}$	 the inverse of  $ {\rm \bf z}$ if and only if:  ${\rm \bf z}\star({\rm \bf z})^{-1}=({\rm \bf z})^{-1}\star{\rm \bf z}={\rm \bf e}$. where ${\rm  unfold}({\rm \bf e})  =(1,0,0\ldots,0)^T$.
   			\item 	An $n_{1}\times n_{1} \times n_{3}$ tensor $\mathscr{A}$ is invertible, if there exists a tensor $\mathscr{B}$ of order  $n_{1}\times n_{1} \times n_{3}$  such that
   			$$\mathscr{A}  \star \mathscr{B}=\mathscr{I}_{ n_{1}  n_{1}  n_{3}} \qquad \text{and}\qquad \mathscr{B}  \star \mathscr{A}=\mathscr{I}_{ n_{1}  n_{1}  n_{3}}.$$
   			In that case, we set $\mathscr{B}=\mathscr{A}^{-1}$. 	It is clear that 	$\mathscr{A}$ is invertible if and only if   ${\rm bcirc}(\mathscr{A})$ is invertible.
   			\item The transpose of $\mathscr{A}$  is obtained by transposing each of the frontal slices and then reversing the order of transposed frontal slices 2 through $n_3$. 
   			\item If $\mathscr {A}$, $\mathscr {B}$ and $\mathscr {C}$ are tensors of appropriate order, then
   			$$(\mathscr {A} \star \mathscr {B}) \star \mathscr {C}= \mathscr {A} \star (\mathscr {B} \star \mathscr {C}).$$
   			\item Suppose $\mathscr {A}$ and $\mathscr {B}$ are two tensors such $\mathscr {A} \star \mathscr {B}$ and $ \mathscr {B}^T \star \mathscr {A}^T$ are  defined. Then  $$(\mathscr {A} \star \mathscr {B})^T= \mathscr {B}^T \star \mathscr {A}^T.$$
   		\end{enumerate}
   	\end{definition}

   	\begin{definition} Let 	$\mathscr{A}$ and $\mathscr{B}$ two tensors in $\mathbb{R}^{n_1 \times n_2 \times n_3}$. Then
   		\begin{enumerate}
   			\item The scalar inner product is defined by
   			$$\langle \mathscr{A}, \mathscr{B} \rangle = \displaystyle \sum_{i_1=1}^{n_1} \sum_{i_2=1}^{n_2}  \sum_{i_3=1}^{n_3} a_{i_1 i_2 i_3}b_{i_1 i_2 i_3}.$$
   			\item The norm of $\mathscr{A}$ is defined by
   			$$ \Vert \mathscr{A} \Vert_F=\displaystyle \sqrt{\langle  \mathscr{A} ,  \mathscr{A}  \rangle}.$$
   		\end{enumerate}
   	\end{definition}
   	\medskip
  
   	\begin{definition}\cite{klimer2}
   		\begin{enumerate}
   		\item An $n_{1}\times n_{1} \times n_{3}$ tensor  $\mathscr{Q}$  is orthogonal if
   		$$\mathscr{Q}^{T}   \star  \mathscr{Q}=\mathscr{Q} \star \mathscr{Q}^{T}=\mathscr{I}_{ n_{1}  n_{1}  n_{3}}.$$		
   	\item 	A tensor is called f-diagonal if its frontal slices are orthogonal matrices. It is called upper triangular if all its frontal slices are upper triangular. 
   		
   		\end{enumerate}
   	\end{definition}
   \begin{definition}\label{bloctens0}\cite{miaoTfunction}{(Block tensor based on T-product)} % voir article Generalized Tensor Function via the Tensor Singular Value page 24
   	Suppose $\mathscr{A}   \in {\mathbb R}^{n_{1}\times m_{1} \times n_{3}} $, $\mathscr{B}\in {\mathbb R}^{n_{1}\times m_{2} \times n_{3}}$, $\mathscr{C}   \in {\mathbb R}^{n_{2}\times m_{1} \times n_{3}} $ and   $\mathscr{D}\in {\mathbb R}^{n_{2}\times m_{2} \times n_{3}}$ are four tensors. The block tensor
   	$$\left[ {\begin{array}{*{20}{c}}
   		{\mathscr{A}}&{\mathscr{B}} \\
   		{\mathscr{C}}&{\mathscr{D}} \\
   		\end{array}} \right]\in {\mathbb R}^{(n_{1}+n_2)\times (m_1+m_{2}) \times n_{3}} $$
   	is defined by compositing the frontal slices of the four tensors.
   \end{definition}
\begin{proposition} \label{propoblock}
	Let $\mathscr{A},\mathscr{A}_1   \in {\mathbb R}^{n \times s \times n_{3}} $, $\mathscr{B},\mathscr{B}_1\in {\mathbb R}^{n \times p \times n_{3}}$, $\mathscr{A}_2   \in {\mathbb R}^{\ell\times s \times n_{3}} $,  $\mathscr{B}_2\in {\mathbb R}^{\ell\times p \times n_{3}}$, $\mathscr{C}    \in {\mathbb R}^{s \times n \times n_{3}}$, $\mathscr{D}    \in {\mathbb R}^{p \times n \times n_{3}}$  and  $\mathscr{F}    \in {\mathbb R}^{n \times n \times n_{3}} $ . Then 
	\begin{enumerate}
		\item $\mathscr {F} \star	\left[\mathscr {A} \;\mathscr {B}\right]=	\left[ \mathscr {F}\star \mathscr {A} \;\;\mathscr {F}\star\mathscr {B}\right]\in \mathbb{R}^{ n\times (s+p)\times n_3  }.$
		\item $\begin{bmatrix}
		\mathscr {C} \\
		\mathscr {D} 
		\end{bmatrix}\star \mathscr {F}=\begin{bmatrix}
		\mathscr {C}\star \mathscr {F} \\
		\mathscr {D} \star \mathscr {F}
		\end{bmatrix}\in \mathbb{R}^{ (s+p)    \times n\times n_3}.$
		\item $\left[ \mathscr {A} \;\mathscr {B}\right]\star\begin{bmatrix}
	\mathscr {C} \\
		\mathscr {D} 
		\end{bmatrix}=\mathscr {A}\star\mathscr {C}+\mathscr {B}\star \mathscr {D}\in \mathbb{R}^{   n\times n \times n_3}.$\\
		\item $\begin{bmatrix}
		\mathcal {A}_1&\mathcal {B}_1 \\
		\mathcal {A}_2&\mathcal {B}_2
		\end{bmatrix}\star  \begin{bmatrix}
		\mathscr {C} \\
		\mathscr {D} 
		\end{bmatrix}=\begin{bmatrix}
		\mathscr {A}_1\star\mathscr {C}+\mathscr {B}_1\star\mathscr {D} \\
		\mathscr {A}_2\star\mathscr {C}+\mathscr {B}_2\star\mathscr {D}
		\end{bmatrix}\in \mathbb{R}^{  (\ell+n)   \times n\times n_3}.$
	\end{enumerate}
\end{proposition}

	\medskip
\noindent Now we introduce the T-diamond tensor product.  
\medskip
  \begin{definition}%\cite{ElIchi}
		Let
		$\mathscr{A} =[\mathscr{A}_{1},\ldots,\mathscr{A}_{p}]\in {\mathbb R}^{n_{1}\times ps \times n_3},$ where $  \mathscr{A}_{i} \in {\mathbb R}^{n_{1}\times s \times n_3} , \, i =1,...,p$ %\hspace{6cm}
		and let $	\mathscr{B} =[\mathscr{B}_{1},\ldots,\mathscr{B}_{\ell }]\in {\mathbb R}^{n_{1}\times \ell s \times n_3}$ with $  \mathscr{B}_{j}\in {\mathbb R}^{n_{1}\times s \times n_3}, \, j =1,...\ell$. Then,  the product $\mathscr{A}^{T} \diamondsuit  \mathcal{B} $ is the   $p \times  \ell  $  matrix  given by : 
		$$ (\mathscr{A}^{T} \diamondsuit  \mathcal{B})_{i,j}  =   \langle \mathscr {A}_i  ,\mathscr{B}_{j} \rangle  \;\;.   
		$$ 
\end{definition}

\medskip
\noindent We consider the  following Sylvester tensor equation 
\begin{equation}\label{eqsyl1}
  \mathscr{A} \star\mathscr{X}+\mathscr{X} \star\mathscr{B}=\mathscr{C},
\end{equation}
where  $\mathscr{A}\in \mathbb{R}^{n \times n \times n_{3}},\mathscr{B}\in \mathbb{R}^{q \times q \times n_{3}}$ and $\mathscr{X},\mathscr{C}\in \mathbb{R}^{n \times q \times n_{3}}$ respectively. $\mathscr{A},\mathscr{B}$ and $\mathscr{C}$ are  given, $\mathscr{X}$ the unknown tensor to  be determined.
\medskip
\noindent  \begin{theorem}
	 Sylvester tensor equation (\ref{eqsyl1}) has a unique solution if and only if $\varGamma({\bf A})\cap\varGamma({\bf B})=\O{}$,   	where the block diagonal matrix $\bf A,\bf B$  are defined by \eqref{dft9}, and $\varGamma({\bf A}), \varGamma({\bf B})$ denotes the set
	 of eigenvalues of the matrix {\bf A} and {\bf B} respectively.
\end{theorem}  
\begin{proof}
	 It is easy to show that the   product  $\mathscr{A} \star\mathscr{X}+\mathscr{X} \star\mathscr{B}=\mathscr{C}$ is equivalent to ${\bf C}= {\bf A}  {\bf X}+{\bf X}  {\bf B}$ ( Sylvester matrix equation). which has a unique solution if and only $\varGamma({\bf A})\cap\varGamma({\bf B})=\O{}$.
\end{proof}

\section{Tensor   tFOM and tGMRES   algorithms}

\subsection{The   tArnoldi method }
Consider  the following tensor linear  system of equations
\begin{equation}\label{syslintenst}
%\mathscr{A}\star \mathscr{X}=\mathscr{C},
   {\mathcal M} (\mathscr{X}) =  \mathscr{C}
\end{equation}  
where $\mathcal{M}$ an linear operator  , $\mathscr{C}$ and $ \mathscr{X}\in \mathbb{R}^{n\times s \times p}$. 
\noindent We introduce the  tensor Krylov subspace $\mathcal{\mathscr{TK}}_m(\mathcal{M},\mathscr{V} )$ associated to the T-product, defined for  the pair $(\mathcal{M},\mathscr{V})$   as follows
\begin{equation}
\label{tr3}
\mathcal{\mathscr{TK}}_m(\mathcal{M},\mathscr{V} )= {\rm Tspan}\{ \mathcal{M}, \mathcal{M} ( \mathscr{V}),\ldots,\mathcal{M}^{m-1} ( \mathscr{V}) \}
\end{equation}
where $\mathcal{M}^{i-1}(\mathscr{V})=\mathcal{M}^{i-2}(\mathcal{M} ( \mathscr{V}))$, for $i=2,\ldots,m$ and $\mathscr{A}^{0}$ is the identity tensor. In the following algorithm, we define the  Tensor   tArnoldi algorithm.

\begin{algorithm}[H]
	\caption{Tensor   tArnoldi algorithm} \label{TGA}
	\begin{enumerate}
		\item 	{\bf Input.} $\mathscr{A}\in \mathbb{R}^{n\times n \times p}$, $\mathscr{V}\in \mathbb{R}^{n\times s \times p}$ and the positive integer $m$.
		\item Set $\beta=\|\mathscr{V}\|_F$, $\mathscr{V}_{1} =   \dfrac{\mathscr{V}}{  \beta}$
		\item For $j=1,\ldots,m$
		\begin{enumerate}
			\item $\mathscr{W}=  \mathcal{M} (    \mathscr{V}_j)$
			\item for $i=1,\ldots,j$
			\begin{enumerate}
				\item $h_{i,j}=\langle \mathscr{V}_i, \mathscr{W} \rangle$
				\item $\mathscr{W}=\mathscr{W}-h_{i,j}\;\mathscr{V}_i$	
			\end{enumerate}	
			\item End for
			\item   $h_{j+1,j}=\Vert \mathscr {W} \Vert_F$. If $h_{j+1,j}=0$, stop; else
			\item $\mathcal {V}_{j+1}=\mathscr {W}/h_{j+1,j}$.

		\end{enumerate}
		\item End and return $\mathbb{V}_m$
	\end{enumerate}
\end{algorithm}

\medskip
It is not difficult to show that after m steps of Algorithm \ref{TGA},  the tensors  $\mathscr{V}_{1},\ldots,\mathscr{V}_{m}$, form an orthonormal basis of the tensor global Krylov  subspace $\mathscr{TK} _{m}(\mathcal{M},\mathscr{V})$.
%	\end{proposition}
\noindent Let $\mathbb{V}_{m}  $ be  the $(n\times (sm)\times p)$ tensor with frontal slices $\mathscr{V}_{1},\ldots,\mathscr{V}_{m}$ and let $ {\widetilde{H}}_{m}$ be  the $(m+1)\times m  $ upper  Hesenberg matrix whose elements are the $h_{i,j}$'s defined by Algorithm \ref{TGA}.    Let  $ {H}_{m}$ be the matrix obtained from $\widetilde{ { H}}_{m}$ by deleting its last row;  $H_{.,j}$ will denote the $j$-th column of the matrix  $H_m$ and $\mathscr{A}\star\mathbb{V}_{m}  $ is  the $(n\times (sm)\times p)$ tensor with frontal slices  $\mathscr{A}\star\mathscr{V}_{1},\ldots,\mathscr{A}\star\mathscr{V}_{m}$: 
\begin{equation}
\label{ev1}
\mathbb{V}_{m}:=\left[  \mathscr{V}_{1},\ldots,\mathscr{V}_{m}\right] \;\;\; {\rm and}\;\;\; \mathscr{W} _{m}:=[\mathcal{M} ( \mathscr{V}_1),\ldots,\mathcal{M} ( \mathscr{V}_m)].
\end{equation}
We introduce the product  $\circledast$ defined by $$\mathbb{V}_{m}\circledast y=\sum_{j=1}^{m} {y}_{j}\mathscr{  {V}}_{j},\; \; y= (y_1,\ldots,y_m)^T\in \mathbb{R}^m,$$
and we set
\begin{equation*}
\mathbb{V}_{m}\circledast {{ {H}}_{m}}=\left[   \mathbb{V}_m\circledast H_{.,1} ,\ldots,\mathscr{V}_{m}\circledast H_{.,m} \right].
\end{equation*}
Then, it is easy to see that for all  vectors  $u$ and $ v$ in  $\mathbb{R}^{m}$, we have 
\begin{equation}\label{relationproduit}
\mathbb{V}_{m}\circledast (u+v)=\mathbb{V}_{m}\circledast u + \mathbb{V}_{m}\circledast v \quad \text{and}\quad (\mathbb{V}_{m}\circledast H_m)\circledast u=\mathbb{V}_{m}\circledast(H_m\;u).
\end{equation}
With these notations, we can show the following result  that will be useful later on.
\medskip

\begin{proposition}\label{normfrobnorm2}
	Let $\mathbb{V}_{m}$ be the tensor defined by  $\left[ \mathscr{V}_{1},\ldots,\mathscr{V}_{m}\right]$ where $\mathscr{V}_{i}\in \mathbb{R}^{n\times s\times p} $ are defined by the Tensor   tArnoldi algorithm. Then, we have
	\begin{equation}
	\|\mathbb{V}_{m}\circledast y\|_F=\|y\|_2, \; \forall y= (y_1,\ldots,y_m)^T\in \mathbb{R}^m .
	\end{equation}
\end{proposition}

\begin{proof}
	From the definition of  the  product  $\circledast$,  we have $\sum_{j=1}^{m} {y}_{j}\mathscr{  {V}}_{j}=\mathbb{V}_{m}\circledast y$. Therefore,  
	$$\|\mathbb{V}_{m}\circledast y\|_F^2= \left< \sum_{j=1}^{m} {y}_{j}\mathscr{  {V}}_{j},\sum_{j=1}^{m} {y}_{j}\mathscr{  {V}}_{j} \right>_F.$$
	But, since the tensors $ \mathscr{  {V}}_{i}$'s are orthonormal, it follows that
	$$\|\mathbb{V}_{m}\circledast y\|_F^2= \sum_{j=1}^{m} {y}_{j}^2=\|y\|_2^2,$$
	which shows the result.
	
\end{proof}

\noindent With the above notations, we can easily prove the  results of the following proposition.

\begin{proposition}\label{T-GlobalArnolproposit}
	Suppose that m steps of Algorithm \ref{TGA} have been   run. Then, the following statements hold:
	\begin{eqnarray} \label{relationtarnoldi}
	\mathscr{W} _{m}&=&\mathbb{V}_{m}\circledast {{ {H}}_{m}} +  h_{m+1,m}\left[  \mathscr{O}_{n\times s\times p},\ldots,\mathscr{O}_{n\times s\times p},\mathscr{V}_{m+1}\right],\\
	\mathscr{W} _{m}&=&\mathbb{V}_{m+1}   \circledast \widetilde{ { H}}_{m}, \\
	\mathbb{V}_{m}^{T}\diamondsuit\mathscr{W} _{m}&=& {H}_{m}, \\	
	\mathbb{V}_{m+1}^{T}\diamondsuit  \mathscr{W} _{m}&=&\widetilde{ { H}}_{m},\\
	\mathbb{V}_{m}^{T} \diamondsuit\mathbb{V}_m&=& {I}_{ m  },
	\end{eqnarray}
	where ${I}_{ m  }$ the identity matrix and $\mathscr{O}$ is the tensor having all its entries equal to zero.
\end{proposition}

\medskip
\begin{proof}
	From Algorithm \ref{TGA}, we have $ \mathcal{M}( \mathscr{V}_{j})=\displaystyle \sum_{i =1}^{j+1}h_{i,j}  \mathscr{V}_{i}$.
	Using the fact that  $$\mathscr{W} _{m}:=[\mathcal{M} ( \mathscr{V}_1),\ldots,\mathcal{M} ( \mathscr{V}_m)],$$  
	the $j$-th frontal slice    of $\mathscr{W} _{m}$ is given by 
	\begin{align*}
	(\mathscr{W}_{m})_j= \mathcal{M} ( \mathscr{V}_j)&=\sum_{i =1}^{j+1}h_{i,j} \mathscr{V}_{i}.	 
	\end{align*}
	Furthermore, from the definition of the $\circledast$ product, we have 
	\begin{align*}
	(\mathbb{V}_{m+1}   \circledast \widetilde{ { H}}_{m})_j&=\mathbb{V}_{m+1}\circledast H_{.,j}\\
	&= \sum_{i =1}^{j+1}h_{i,j} \mathscr{V}_{i},
	\end{align*}
	which proves the first two relations. The other relations follow from the definition of T-diamond product
\end{proof} 
\subsection{The   tFOM method }
In the following, we  examined the tensor  full orthogonalization  (tFOM) method . It could be considered as generalization of the  global FOM algorithm \cite{jbilou1}. 
Let $\mathscr{  {X}}_{0}\in \mathbb{R}^{n\times s\times p}$ be an arbitrary initial guess with   the corresponding  residual
$\mathscr{R}_0=\mathscr{C}-\mathcal{M} ( \mathscr{X}_0)$.    The aim of tensor  T-global GMRES method is to find and approximate solution  $\mathscr{X}_{m}$ approximating the exact solution $\mathscr{X}^*$ of \eqref{syslintenst}  such that 
\begin{equation}
\label{fom1}
\mathscr{X}_{m}-\mathscr{X}_{0}= \mathscr{P}_{m}\in \mathscr{TK}_{m}(\mathcal{M},\mathscr{R}_0), 
\end{equation}   and 
\begin{equation}
\label{fom2} 
 \mathscr{R}_{m}  \perp \mathscr{TK} _{m}(\mathcal{M},\mathscr{R}_0) 
\end{equation} 

$\mathscr{P}_{m}\in \mathscr{TK}_{m}(\mathcal{M},\mathscr{R}_0)$ can be expressed  as   $\mathscr{P}_{m}=\mathscr{V}_{m}\circledast y$ with $ y= (y_1,\ldots,y_m)^T\in \mathbb{R}^m$. From where the residual $\mathscr{R}_{m}$  is given by
\begin{equation}
\label{resfom2} 
\mathscr{R}_{m}  = \mathscr{R}_{0}-   \mathscr{W}_{m}\circledast y
\end{equation} 
  Using (\ref{resfom2}),  the  relation (\ref{fom2}) can be expressed as 
  \begin{equation}
  \label{resfomortho} 
  \left\langle \mathscr{V}_{i},\mathscr{R}_{m} \right\rangle =\left\langle  \mathscr{V}_{i},   \mathscr{W}_{m}\circledast y \right\rangle , \;\;  i=1,\ldots,m
  \end{equation} 
   where  the tensors  $\mathscr{V}_{1},\ldots,\mathscr{V}_{m}$, form an orthonormal basis of the tensor global Krylov  subspace $\mathscr{TK} _{m}(\mathcal{M},\mathscr{R}_0)$ .
   Using (\ref{relationproduit}) and  (\ref{relationtarnoldi}), equation (\ref{resfomortho}) can be expressed as:
   \begin{equation}
   \label{fomrelation} 
     {H}_m\;y=\|\mathscr{R}_{0}\|_F\; e_1^{(m)},
   \end{equation} 
   where $e_1^{(m)}$ is the first canonical basis vector in $\mathbb{R}^m$ and  ${H}_m$ the  Hessenberg matrix of  size $(m\times m)$ obtained from \Cref{TGA}.  
   
   	\begin{proposition}
   	At step $m$, the  norm of the residual $\mathscr{R}_{m}=\mathscr{C}- \mathcal{M} (  \mathscr{X}_{m})=\mathscr{R}_{0}-   \mathscr{W}_{m}\circledast y_m$ produced by the tFOM method for tensor equation (\ref{eq1}) has the following expression	 
   	\begin{equation}\label{resnrmfom}
   	\left\|\mathscr{R}_{m}\right\|_{F}=h_{m+1,m}\left|y_m^{(m)}\right|.
   	\end{equation}
   	 where $y_m^{(m)}$ is the last component of the vector $y_m$ 
   \end{proposition}
   \begin{proof}
   	At step $m$, using the relations (\ref{resfom2}) and      (\ref{fom2}) the norm of the  residual $\mathscr{R}_m$ can be expressed as
   	$$\left\|\mathscr{R}_{m}\right\|_{F}=\left\|\mathscr{R}_{0}-\mathbb{V}_{m}\circledast { ({H}_{m}\,y_m)} +  h_{m+1,m}\left[  \mathscr{O}_{n\times s\times p},\ldots,\mathscr{O}_{n\times s\times p},\mathscr{V}_{m+1}\right]\circledast y_m\right\|_{F} .$$
   	since ${H}_m\;y_m=\|\mathscr{R}_{0}\|_F\; e_1^{(m)}$ and $\mathscr{R}_{0}\|_F (\mathbb{V}_{m}\circledast e_1^{(m)})$, we get 
   	$$\left\|\mathscr{R}_{m}\right\|_{F}= h_{m+1,m}\left\|\left[  \mathscr{O}_{n\times s\times p},\ldots,\mathscr{O}_{n\times s\times p},\mathscr{V}_{m+1}\right]\circledast y_m\right\|_{F}=h_{m+1,m}\left|y_m^{(m)}\right| $$
   	 	which shows the results. 
   \end{proof}

\subsection{The   tGMRES method }
\medskip
\noindent In the sequel, we develop the tensor    tGMRES algorithm  for solving the problem \eqref{syslintenst}. It could be considered as generalization of the  global GMERS algorithm \cite{jbilou29}. 
Let $\mathscr{  {X}}_{0}\in \mathbb{R}^{n\times s\times p}$ be an arbitrary initial guess with   the corresponding  residual
$\mathscr{R}_0=\mathscr{C}-\mathcal{M} ( \mathscr{X}_0)$.    The purpose of tensor    tGMRES method is to find and approximate solution  $\mathscr{X}_{m}$ approximating the exact solution $\mathscr{X}^*$ of \eqref{syslintenst}  such that 
\begin{equation}
\label{gmres1}
\mathscr{X}_{m}-\mathscr{X}_{0} \in \mathscr{TK}_{m}(\mathcal{M},\mathscr{R}_0), 
\end{equation} 
with the classical  minimization property 
\begin{equation}
\label{gmres2} 
\Vert \mathscr{R}_{m}\Vert_F = \displaystyle \min_{ \mathscr{X} \in \mathscr{X}_{0} + \mathscr{TK} _{m}(\mathcal{M},\mathscr{R}_0)}  \|\mathscr{C}-\mathcal{M}(  \mathscr{X})\|_F.
\end{equation} 

\noindent Let   $\mathscr{X}_{m}=\mathscr{X}_{0}+\mathbb{V}_{m}\circledast y  $ with $ {y} \in \mathbb{R}^m $,   be the approximate solution satisfying \eqref{gmres1}. Then, 
\begin{align*}
\mathscr{R}_m=&\mathscr{C}-\mathcal{M} ( \mathscr{X}_m)\\
=& \mathscr{C}-\mathcal{M}  \left(\mathscr{X}_{0}+ \mathbb{V}_{m}\circledast y\right) \\
 =&\mathscr{R}_{0}-  \mathbb{W}_{m} \circledast y. 
\end{align*}
It follows then that 
\begin{align*}
\|\mathscr{R}_m \|_{F}&=  \displaystyle \min_{ y\in \mathbb{R}^{m }}
\|\mathscr{R}_{0}-  \mathbb{W}_{m} \circledast y\|_F,
\end{align*}
where $\mathscr{W} _{m}:=[\mathcal{M} ( \mathscr{V}_1),\ldots,\mathcal{M} ( \mathscr{V}_m)]$ is the $(n\times sm\times p)$ tensor defined earlier.  Using  Proposition  \ref{normfrobnorm2} and the fact that      $\mathscr{R}_{0}=\|\mathscr{R}_{0}\|_F \mathscr{V}_1 $  with  $\mathscr{V}_1 = \mathscr{V}_{m+1}\circledast e_{1}$, where $e_{1}$  the first  canonical basis vector in $\mathbb{R}^{m+1}$,  we get 
\begin{align*}
\|\mathscr{R}_{0}-(\mathscr{A}\star\mathbb{V}_{m})\circledast y\|_F&=\| \mathscr{R}_{0}- (\mathbb{V}_{m+1}   \circledast \widetilde{ { H}}_{m}) \circledast y  \|_F\\
&=\|\|\mathscr{R}_{0}\|_F (\mathbb{V}_{m+1}\circledast e_1)- (\mathbb{V}_{m+1}   \circledast \widetilde{ { H}}_{m}) \circledast y \|_F\\
&=\| \mathbb{V}_{m+1}\circledast (||\mathscr{R}_{0}\|_F e_1-\widetilde{ { H}}_{m}  y) \|_F \\
&=\|\; \|\mathscr{R}_{0}\|_F\; e_1-\widetilde{ { H}}_{m}   y \|_2.
\end{align*}
Finally, we obtain  
\begin{equation}\label{solutdegmresxm}
\mathscr{X}_{m}=\mathscr{X}_{0}+ \mathbb{V}_{m} \circledast y, 
\end{equation}
where,
\begin{equation}\label{Gmressol}
y=  \text{arg } \min_{ {  {y}}\in \mathbb{R}^{m }}||\; ||\mathscr{R}_{0}||_F\; e_1-\widetilde{ { H}}_{m}  y) ||_2.
\end{equation}

	\begin{proposition}
	At step $m$, the residual $\mathscr{R}_{m}=\mathscr{C}-\mathcal{M} ( \mathscr{X}_m)$ produced by the tensor Global GMRES method for tensor equation (\ref{eq1}) has the following expression
	\begin{equation}\label{resex}
	\mathscr{R}_m=\mathbb{V}_{m+1} \circledast\left(\gamma_{m+1}Q_me_{m+1}\right),
	\end{equation}
	
	where $Q_m$ is the unitary matrix obtained from the QR decomposition of the upper Hessenberg matrix $\widetilde{H}_{m}$ and $\gamma_{m+1}$ is the last component of the vector $\left\|\mathscr{R}_{0}\right\|_{F} Q_{m}^{\mathrm{T}} e_{1}$ and $e_{m+1}=(0,0, \ldots, 1)^{\mathrm{T}} \in \mathbb{R}^{m+1}.\\$
	Furthermore,
	\begin{equation}\label{resnrm}
	\left\|\mathscr{R}_{m}\right\|_{F}=\left|\gamma_{m+1}\right|.
	\end{equation}
	
\end{proposition}
\begin{proof}
	At step $m$, the residual $\mathscr{R}_m$ can be expressed as
	$$\mathscr{R}_m=\mathbb{V}_{m+1} \circledast\left(\beta e_{1}-\widetilde{H}_{m} y_{m}\right),$$
	by considering the QR decomposition $\widetilde{H}_{m}=Q_{m}\widetilde{U}_m$ of the $(m + 1) \times m$ matrix $\widetilde{H}_{m}$, we get
	$$\mathscr{R}_m=\left(\mathbb{V}_{m+1} \circledast Q_m\right)\circledast\left(\beta Q_m^T e_{1}-\widetilde{U}_{m} y_{m}\right).$$
	Since $y$ solves problem (\ref{Gmressol}), it follows that
	$$\mathscr{R}_m=\mathbb{V}_{m+1} \circledast\left(\gamma_{m+1}Q_me_{m+1}\right),$$
	where $\gamma_{m+1}$ is the last component of the vector $\beta Q_{m}^T e_{1}.$ Therefore,
	\begin{eqnarray*}
		\left\|\mathscr{R}_{m}\right\|_{F}&=&\left\|\mathbb{V}_{m+1} \circledast\left(\gamma_{m+1}Q_me_{m+1}\right)\right\|_F\\
		&=&\left\|\gamma_{m+1}Q_me_{m+1}\right\|_2\\
		&=&\left|\gamma_{m+1}\right|,
	\end{eqnarray*}
	which shows the results.  
\end{proof}

 \begin{algorithm}[h!]
	\caption{ tFOM and tGMRES Algorithms }  \label{GlobaFOMGMRES}
	\begin{enumerate}
		\item 	{\bf Input.} $\mathscr{A} \in \mathbb{R}^{n\times n \times n_3}$   , $\mathscr{B} \in \mathbb{R}^{s\times s \times n_3}$, $\mathscr{X}_{0},\mathscr{C}\in \mathbb{R}^{n\times s \times n_3}$, the maximum number of iteration  $\text{Iter}_{\text{max}} $  an integer $m$ and a tolerance  $tol$.
			\item 	{\bf Output.} $\mathscr{X}_m\in \mathbb{R}^{n\times s \times n_3}$ the approximate solution of  (\ref{syslintenst}).
		\item  Compute $\mathscr{R}_0=\mathscr{C}-\mathcal{M}( \mathscr{X}_0) $.
		\item For $k=1,\ldots,\text{Iter}_{\text{max}}$
		\begin{enumerate}
			\item  Apply Algorithm \ref{TGA} to compute  $\mathbb{V}_{m}$ and  $ \widetilde{ { H}}_{m}$.
			
			\item Solve the problem  : $\;\;$
			$\begin{cases}
				  {H}_m\;y_m=\|\mathscr{R}_{0}\|_F\; e_1^{(m)}\;\; ({\rm \textbf{tFOM method}})\\
				  y_m=  \text{arg } \min_{ {  {y}}\in \mathbb{R}^{m }}||\; ||\mathscr{R}_{0}||_F\; e_1^{(m+1)}-\widetilde{ { H}}_{m}  y_m ||_2. \;{\rm(\textbf{ tGMRES method  } )}
				 	\end{cases}$
			 
			\item Compute  $\mathscr{X}_{m}=\mathscr{X}_{0}+ \mathbb{V}_{m} \circledast y_m  $
		\end{enumerate}
		\item If $||\mathscr{R}_{m}||_F<tol$ \textbf{then}
		
		$\;\;\;\;\;\;$\textbf{return} $\mathscr{X}_{m}$;
		            
		\item else   $\mathscr{X}_{0}=\mathscr{X}_{m}$ and go to Step 2.
		\item 	{\bf Output.} $\mathscr{X}_m\in \mathbb{R}^{n\times s \times n_3}$ the approximate solution of  (\ref{syslintenst}).   
		
	\end{enumerate}
\end{algorithm}

\section { t-Bartels-Stewart method for Sylvester tensor equations of small size}
In this section, we introduce the  \textit{t-Bartels-Stewart} method for Sylvester tensor equations of small size  based on T-product formalism , as  a  generalisation of  the well known  Bartels and Stewart algorithm proposed in \cite{Bartels}. Motivated by
the matrix case, the  \textit{t-Bartels-Stewart} method is based on transforming the tensors $\mathscr{A}$ and $\mathscr{B}$ into \textit{t-real Schur} form that will be defined later-. This gives an new triangular tensor  equation that will be solved by  \textit{t-back-substitution} method.  

\noindent First, we introduce the \textit{t-real schur} decomposition.
  \begin{theorem}
	 Let  $\mathscr{A}\in \mathbb{R}^{n \times n \times n_{3}}$. Then  $\mathscr{A}$ can be factored as
	 \begin{equation}\label{Schur}
	  \mathscr{A}=\mathscr{U}\star\mathscr{R}\star\mathscr{U}^T,
	 \end{equation}
	 where $\mathscr{R}\in \mathbb{R}^{n \times n \times n_{3}}$   quasi upper triangular tensor  (each frontal slice of $\mathscr{R}$ is   quasi upper triangular) and $\mathscr{U}\in \mathbb{R}^{n \times n \times n_{3}}$ orthogonal tensor.
\end{theorem}
\begin{proof}
	 We  have $$(F_{n_3} \otimes I_{n_1})\, {\rm bcirc}(\mathscr {A})\, 	(F_{n_3}^{*} \otimes I_{n_1})={\bf A} = \left(
	 \begin{array}{cccc}
	 	\widetilde {\mathscr {A}}^{(1)}& &&\\
	 	& \widetilde {\mathscr {A}}^{(2)}&&\\
	 	&&\ddots&\\
	 	&&&\widetilde {\mathscr {A}}^{(n_3)}\\
	 \end{array}
	 \right),$$
	 Next, we compute the schur matrix decomposition of each 
	 frontal slice $	\widetilde {\mathscr {A}}^{(i)}, i=1,\ldots,n_3$. Then 
	\begin{equation}\label{decompshur}
	  \left(
	 \begin{array}{cccc}
	 	\widetilde {\mathscr {A}}^{(1)}& &&\\
	 %	& \widetilde {\mathscr {A}}^{(2)}&&\\
	 	&&\ddots&\\
	 	&&&\widetilde {\mathscr {A}}^{(n_3)}\\
	 \end{array}
	 \right)= \left(
	 \begin{array}{cccc}
	 \widetilde {\mathscr {U}}^{(1)}& &&\\
	 %& \widetilde {\mathscr {A}}^{(2)}&&\\
	 &&\ddots&\\
	 &&&\widetilde {\mathscr {U}}^{(n_3)}\\
	 \end{array}
	 \right)\left(
	 \begin{array}{cccc}
	 \widetilde {\mathscr {R}}^{(1)}& &&\\
	 %& \widetilde {\mathscr {A}}^{(2)}&&\\
	 &&\ddots&\\
	 &&&\widetilde {\mathscr {R}}^{(n_3)}\\
	 \end{array}
	 \right)\left(
	 \begin{array}{cccc}
	 \widetilde {\mathscr {U}}^{(1)T}& &&\\
	 %& \widetilde {\mathscr {A}}^{(2)}&&\\
	 &&\ddots&\\
	 &&&\widetilde {\mathscr {U}}^{(n_3)T}\\
	 \end{array}
	 \right) \end{equation}
	 Since $(F_{n_3}^{*} \otimes I_{n_1})
	 \left(
	 \begin{array}{cccc}
	 \widetilde {\mathscr {U}}^{(1)}& &&\\
	 %& \widetilde {\mathscr {A}}^{(2)}&&\\
	 &&\ddots&\\
	 &&&\widetilde {\mathscr {U}}^{(n_3)}
	  \end{array}\right)(F_{n_3}  \otimes I_{n_1})$,$(F_{n_3}^{*} \otimes I_{n_1})
	  \left(
	  \begin{array}{cccc}
	  \widetilde {\mathscr {R}}^{(1)}& &&\\
	  %& \widetilde {\mathscr {A}}^{(2)}&&\\
	  &&\ddots&\\
	  &&&\widetilde {\mathscr {R}}^{(n_3)}
	  \end{array}\right)(F_{n_3}  \otimes I_{n_1})$,$(F_{n_3}^{*} \otimes I_{n_1})
	  \left(
	  \begin{array}{cccc}
	  \widetilde {\mathscr {U}}^{(1)T}& &&\\
	  %& \widetilde {\mathscr {A}}^{(2)}&&\\
	  &&\ddots&\\
	  &&&\widetilde {\mathscr {U}}^{(n_3)T}
	  \end{array}\right)(F_{n_3}  \otimes I_{n_1})$   are block circulant matrices,  by applying
	  the appropriate matrices $(F_{n_3}^{*} \otimes I_{n_1}),(F_{n_3}  \otimes I_{n_1})$ to the left and right   of each matrix in (\ref{decompshur}) respectively, and  folding up the result. This give the result. 
\end{proof}

\medskip
\noindent The \textit{t-real schur} decomposition  can be computed by  the     Algorithm \ref{t-schur}.\\ 
 \begin{algorithm}[!h]
	\caption{t-real schur decomposition}\label{t-schur}
	\begin{enumerate}
		\item 	{\bf Input:} $\mathscr{A}\in {\mathbb R}^{n \times n  \times n_{3}} $  .
		\item \	{\bf Output:} $ \mathscr{R}\in {\mathbb R}^{n\times n \times n_{3}} $ upper triangular tensor.
		 $ \mathscr{U}\in {\mathbb R}^{n\times n \times n_{3}} $ orthogonal tensor.
		\item Set $\widetilde{\mathscr{ {A}}}=\text{{\tt fft}}(\mathscr{ A},[],3) $   
		\begin{enumerate}
			\item for $i=1,\ldots,n_3$
			\begin{enumerate}
				\item $	[ \widetilde{\mathscr{ {Q}}}^{(i)} \widetilde{\mathscr{ {R}}}^{(i)}] =      \text{schur}( \widetilde{\mathscr{ {A}}}^{(i)}) $ \quad (real schur  matrix decomposition)
			\end{enumerate}	
			\item End
		\end{enumerate}
		\item $  \mathscr{Q}  =\text{{\tt ifft}} ( \widetilde{\mathscr{ {Q}}},[],3),   \mathscr{R}  =\text{{\tt ifft}} ( \widetilde{\mathscr{ {R}}},[],3) $
		\item End
	\end{enumerate}
\end{algorithm}
\medskip
\noindent  Next, we transform the tensors $\mathscr{A}$ and $\mathscr{B}$ into \textit{t-real schur} forms $ \mathscr{A}=\mathscr{U_A}\star\mathscr{R_A}\star\mathscr{U_A}^T$
 and $ \mathscr{B}=\mathscr{U_B}\star\mathscr{R_B}\star\mathscr{U_B}^T$. Then the Sylvester tensor equation (\ref{eqsyl1}) become 
 \begin{align*}
   (\mathscr{U_A}\star\mathscr{R_A}\star\mathscr{U_A}^T)\star\mathscr{X}+\mathscr{X}\star(\mathscr{U_B}\star\mathscr{R_B}\star\mathscr{U_B}^T)=\mathscr{C} 
 \end{align*}
 which equivalent to
 \begin{align*}
 \mathscr{U_A}^T\star\left[ (\mathscr{U_A}\star\mathscr{R_A}\star\mathscr{U_A}^T)\star\mathscr{X}+\mathscr{X}\star(\mathscr{U_B}\star\mathscr{R_B}\star\mathscr{U_B}^T)\right] \star\mathscr{U_B} =\mathscr{U_A}^T\star \mathscr{C} \star\mathscr{U_B}
 \end{align*}
Since $\mathscr{U_A}$ and $\mathscr{U_B}$ are orthogonal tensors, we get 
\begin{equation}\label{sylvtriang}
 \mathscr{R_A}\star\mathscr{Y} + \mathscr{Y}\star\mathscr{R_B}  =  \mathscr{C}_1  
\end{equation}
where $ \mathscr{Y}=\mathscr{U_A}^T\star\mathscr{X}\star\mathscr{U_B}$ and $\mathscr{C}_1=\mathscr{U_A}^T\star \mathscr{C} \star\mathscr{U_B}$. 
   		  \begin{remark}
   			 The equation (\ref{sylvtriang}) can be solved by the  \textit{t-back-substitution} method. This 
   			 method follow the same steps of  matrix back substitution, where the tube fibers (3-mode fibers), lateral slices and T-product play the role of  scalars, vectors and matrix product respectively.
   			 \end{remark}      		 
   		\noindent Finally, the \textit{t-Bartels-Stewart} method
   		is   implemented by	Algorithm \ref{t-bartels}. \\
   			\begin{algorithm}[!h]
   			\caption{t-Bartels-Stewart method }\label{t-bartels}
   			\begin{enumerate}
   				\item 	{\bf Input:} $\mathscr{A}\in \mathbb{R}^{n \times n \times n_{3}},\mathscr{B}\in \mathbb{R}^{q \times q \times n_{3}}$ and $ \mathscr{C}\in \mathbb{R}^{n \times q \times n_{3}}$  .
   				\item \	{\bf Output:}   $\mathscr{X}_S\in \mathbb{R}^{n \times q \times n_{3}}$ solution of Sylvester tensor equation (\ref{eqsyl1}) of small size.
   				\item Compute \textit{t-real schur} forms $$ \mathscr{A}=\mathscr{U_A}\star\mathscr{R_A}\star\mathscr{U_A}^T,\quad  \mathscr{B}=\mathscr{U_B}\star\mathscr{R_B}\star\mathscr{U_B}^T$$ using Algorithm \ref{t-schur}.
   				\item  Compute $\mathscr{Y}_S$ solution of tensor equation (\ref{sylvtriang}) using \textit{t-back-substitution} method.
   				\item Compute $\mathscr{X}_S=\mathscr{U_A} \star\mathscr{Y}_S\star\mathscr{U_B}^T$.  
   			 	\end{enumerate}
   		\end{algorithm}
  
   	\section{  Tensor Krylov   methods via T-product for solving large Sylvester tensor equations  }
   	\medskip
   	\noindent In this section, we consider the case when the Sylvester tensor equation  (\ref{eq1}) and
   	(\ref{eq2}) is of large size. As in matrix case, iterative projection methods have been developed; see \cite{elguennouni,Hu,saad}. These methods use Galerkin projection methods, such the classical and the block Arnoldi techniques, to produce low-dimensional Sylvester matrix equations that are solved by using direct methods.\\
   	\noindent In tensor case, the main idea is to transform the large Sylvester tensor equations   using the well-know  Tubal-Block-Arnoldi  that will be introduced later    into low dimensional equations.
   	\subsection{ Tubal Block  Arnoldi method  }
   	\noindent In this subsection, we  introduce  the  Tubal-Block-Arnoldi  method via T-product   as a generalization of the well-know  block Arnoldi method, see \cite{elguennouni} for more details.\\
   	%\noindent First, we introduce  the block Krylov 
   	 	Let  $\mathcal {A} \in \mathbb{R}^{n_{1}\times   n_{1} \times n_{3}}$ be a square tensor and   $\mathscr {V} \in \mathbb{R}^{n_{1} \times s \times I_{3} }, s<<n_1$. The  tensor block Krylov subspace is defined by
   	\begin{equation}\label{krylovblock}
   	\mathcal {K}_m^{Block}(\mathscr {A}, \mathscr{V})=Range\left( \left[  \mathscr {V},\mathscr {A}\star \mathscr {V},\ldots,\mathscr{A}^{m-1}\star \mathscr{V}  \right]\right)    \subset \mathbb{R}^{n_{1}\times s\times n_{3} }
   	\end{equation}
   	where $\mathscr {A}^i =\mathscr {A} \star \mathscr {A}^{i-1} , i=1,\ldots,m-1 $, $\mathscr {A}^0=\mathscr {I}_{n_1n_1n_3}$ and    $$   Range(\mathscr {Z})=\left\lbrace \overrightarrow{\mathscr{Y}}\in \mathbb{R}^{n_{1}\times1\times n_{3}} \mid \overrightarrow{\mathscr{Y}} =\mathscr {Z} \star \overrightarrow{\mathscr{X}} ,\;  \overrightarrow{\mathscr{X}}\in \mathbb{R}^{s\times 1 \times n_{3}} \right\rbrace   $$    	for $\mathscr {Z} \in \mathbb{R}^{n_{1} \times s\times   \times n_{3}}$.\\
   	\noindent Before describing the Tubal-Block-Arnoldi process, let us first introduce the Tubal-QR Factorization.  
   	\noindent  \begin{theorem}(Tubal-QR Factorization)
   		Let $\mathcal{A}=\left[  \overrightarrow{\mathscr{A}}_1,\ldots,\overrightarrow{\mathscr{A}}_{m}\right] \in \mathbb{R}^{n_{1}\times m\times   n_{3}} , \overrightarrow{\mathscr{A}}_i\in \mathbb{R}^{n_{1}\times  1\times   n_{3}}$ for $i=1,\cdots,m$, 
   		there exist orthogonal tensors  $\mathcal{U}\in \mathbb{R}^{n_{1}\times  n_{2}\times   n_{3}} $   and $\mathcal{R}\in \mathbb{R}^{n_{2}\times  n_{2}   \times   n_{3}} $ triangular tensor  such that 
   		\begin{align}\label{QR}
   		\mathcal{A}=\mathcal{Q}\star\mathcal{R} 
   		\end{align}
   		We call  (\ref{QR}) the tubal  QR factorization  of the tensor $\mathcal{A}$.
   	\end{theorem}
   	 \noindent The Tubal QR factorization   is summarized in  Algorithm \ref{TGS}.\\
   	\begin{algorithm}[!h]
   	\caption{ Tubal-QR Factorization (Tubal-QR)} \label{TGS}
   	\begin{enumerate}
   		\item 	{\bf Input.} $\mathscr{A}\in \mathbb{R}^{n\times m \times n_3}$   		\item Set $[\mathscr{V}_{1},R_{1,1,:}]=  Normalization1(\overrightarrow{\mathscr{A}}_1)$
   		\item For $j=1,\ldots,m$
   		\begin{enumerate}
   			\item $\mathscr{W}= \overrightarrow{\mathscr{A}}_j$,
   			\item for $i=2,\ldots,j-1$
   			\begin{enumerate}
   				\item $\mathscr{R}_{i,j,:}=    \mathscr{V}_i ^T\star \mathscr{W}  $
   				\item $\mathscr{W}=\mathscr{W}-  \mathscr{V}_i\star \mathscr{R}_{i,j,:}$	
   			\end{enumerate}	
   			\item End for
   			\item $[\mathscr{Q}_{j },\mathscr{R}_{j ,j,:}]=  Normalization1(\overrightarrow{\mathscr{W}})$
   		\end{enumerate}
   		\item End
   %	\end{enumerate}
   	\item 	{\bf Output.} 
   	$\mathscr{Q}\in \mathbb{R}^{n\times m \times n_3}$  orthogonal and $\mathscr{R}\in \mathbb{R}^{m\times m \times n_3}$
   	such that :  $\mathcal{A}=\mathcal{Q}\star\mathcal{R} $.
	\end{enumerate}   
\end{algorithm}
   \noindent The function  $Normalization1(\overrightarrow{\mathscr{A}})$ described in \cite{klimer3} allow us to write a given tensor $\overrightarrow{\mathscr{A}}\in \mathbb{R}^{n\times 1 \times n_3}$ as follow :
   $ \overrightarrow{\mathscr{A}}=\overrightarrow{\mathscr{U}}\star{\rm \bf a}$ where ${\rm \bf a}\in \mathbb{R}^{1\times 1 \times n_3}$ is invertible and  $\overrightarrow{\mathscr{U}}^T\star \overrightarrow{\mathscr{U}}={\rm \bf e }$.   \\
   
   	\noindent Tubal-Block-Arnoldi method  have to build an orthonormal basis of $\mathbb{V}^b_m$ of the Krylov subspace  $\mathcal {K}_m^{Block}(\mathscr {A}, \mathscr{V})$.
   
   \begin{algorithm}[h!]
   	\caption{The  Tubal Block  Arnoldi Algorithm    (TBA)}\label{TBA}
 
   	\begin{enumerate}
   	\item 	{\bf Input.}  $\mathcal {A} \in \mathbb{R}^{n_{1}\times   n_{1} \times n_{3}}$ be a square tensor,  $\mathscr {V} \in \mathbb{R}^{n_{1} \times s \times I_{3} }, s<<n_1$ and an integer m.
   		
   		\item Set $[\mathscr{V}^b_{1},\mathscr{H} _{0}]=\text{ Tubal-QR}(\mathscr {V} )$.
   		%\item $\mathcal{V}(:,1:s ,:)=
   		\item For j = 1,...,m
   		\begin{enumerate}
   		  % \item  $J=s (j-1)+1:j s$
   		 \item   $\mathcal{W} =\mathcal{A}\star \mathcal{V}_j$    						%\item $H_j=\mathbb{V}_j^T\widetilde{V}_{j+1}$
   			\item   For i = 1,...,j
   			    	\begin{enumerate}
   				\item  $\mathcal{H}_{i,j}=V_{i}^{bT}\star  \mathcal{W}
   				$ 
   				\item  $\mathcal{W} = \mathcal{W} - {V}_{i}^b\star \mathcal{H}_{i,j}$,
   			\end{enumerate}
   			\item End.
   			%\item $\widetilde{V}_{j+1} = \widetilde{V}_{j+1} -\mathbb{V}_j H_j,$
   			\item
   			$\left[ V_{j+1}, \mathcal{H}_{j+1,j}\right]=\text{ Tubal-QR}(\mathcal{W})$. 
   		\end{enumerate}
   		\item End return $\mathbb{V}^b_m$ 
   	\end{enumerate}
   \end{algorithm}
\noindent Using Definition \ref{bloctens0}, we can   put away the  tensors  $\mathcal{H}_{i,j}\in \mathbb{R}^{s\times   s \times n_{3}}$
 into a block tensors  $\mathbb{H}_m $ and  $\mathbb{H}_{m+1} $ defined as follow 
 \begin{equation*}
    \mathbb{H}_{m+1}=\left[ \begin{array}{*{20}{c}}
    \mathcal{H}_{1,1}&{{\mathcal{H}_{1,2} }}&\cdot &\mathcal{H}_{1,m}   \\
    \mathcal{H}_{2,1}&\mathcal{H}_{2,2}&\cdots&\mathcal{H}_{2,m} \\
    &\ddots&\ddots& \vdots\\
    &   &\mathcal{H}_{m,m-1} &  \mathcal{H}_{m,m}\\
    &    &       &    \mathcal{H}_{m+1,m}
    \end{array}  \right] \in  \mathbb{R}^{(m+1)s\times   ms \times n_{3}} ,
    \end{equation*} 
     \begin{equation*}
     \mathbb{H}_m=\left[ \begin{array}{*{20}{c}}
    \mathcal{H}_{1,1}&{{\mathcal{H}_{1,2} }}&\cdot &\mathcal{H}_{1,m}   \\
    \mathcal{H}_{2,1}&\mathcal{H}_{2,2}&\cdots&\mathcal{H}_{2,m} \\
    &\ddots&\ddots& \vdots\\
    &   &\mathcal{H}_{m,m-1} &  \mathcal{H}_{m,m} 
    \end{array}  \right] \in  \mathbb{R}^{ms\times   ms \times n_{3}}.
 \end{equation*}   
   \noindent It is not difficult to show that after m steps of Algorithm \ref{TBA},  the tensor   $\mathbb{V}_{m}^b:=[\mathscr{V}^b_{1},\ldots,\mathscr{V}^b_{m}]\in  \mathbb{R}^{n_1\times   ms \times n_{3}}$, where $\mathscr{V}^b_{i}\in \mathbb{R}^{n_1\times   s \times n_{3}}$  form an orthonormal basis of the tensor Block  Krylov  subspace   
   $\mathcal {K}_m^{Block}(\mathscr {A}, \mathscr{V})$.\\
   \noindent It is easy to see that $\mathbb{H}_m$ can be obtained from 
   $\mathbb{H}_{m+1}$ by deleting the last block row 
    $$[\mathscr{O} _{ssn_3},\ldots,\mathscr{O} _{ssn_3},\mathscr{H} _{m+1,m}] =\mathcal {H}_{m+1,m}\star \mathbb{E}_{m} \;\in  \mathbb{R}^{s \times   ms \times n_{3}}$$ where $\mathscr{O}_{ssn_3}$ denote the zeros tensors of size $(s\times s\times n_3)$ which all
its entries are equal to zeros, and  $\mathbb{E}_{m}=\left[\mathscr{O} _{ssn_3},\ldots,\mathscr{O} _{ssn_3},\mathscr{I} _{ssn_3}\right]\in \mathbb{R}^{s \times   ms \times n_{3}} $ \\
  \noindent Let   
  \begin{align*}\mathscr{A}\star \mathbb{V}^b_{m}:&=[\mathscr{A}\star  \mathscr{V}_{1} ,\ldots, \mathscr{A}\star \mathscr{V}_{m}]\in \mathbb{R}^{n_1 \times   ms \times n_{3}},\\
  \mathbb{V}_{m+1}^b:&=\left[\mathbb {V}_m^b , \mathcal {V}_{m+1}^b\right]\in  \mathbb{R}^{n_1\times   (m+1)s \times n_{3}},\\
 % \mathbb{E}_{m} &=\left[\mathscr{O} _{ssn_3},\ldots,\mathscr{O} _{ssn_3},\mathscr{H} _{m+1,m}]\right\in  \mathbb{R}^{s \times   ms \times n_{3}}\\
  \mathbb{H}_{m}&= (\mathcal {H}_{i,j})_{1\leq i,j\leq m}\in\mathbb{R}^{ ms \times   \times  ms\times n_3},\\
  {\mathbb{H}}_{m+1}&=\begin{bmatrix}
  	\mathbb {H}_m \\
  	\mathcal {H}_{m+1,m}\star \mathbb{E}_{m} 
  \end{bmatrix}\in\mathbb{R}^{(m+1)s\times   ms \times n_{3} }.\\
  \end{align*}

  \noindent With the above notations, we can easily prove the  results of the following proposition.
   \begin{proposition}\label{TBarnoldi}
   	From Algorithm \ref{TBA}, we have
   	 \begin{equation}\label{for1}
   	  \mathcal{A}\star\mathbb{V}_{m}^b= \mathbb{V}^b_{m}\star \mathbb{H}_{m}+ \mathcal{V}^b_{m+1}\star(\mathscr{H}_{m+1,m}\star \mathbb{E}_{m} ),
   	 \end{equation}
   	 \begin{equation}\label{for2}
   	 \mathbb{V}_{m}^{bT}\star\mathcal{A}\star\mathbb{V}^b_{m}= \mathbb{H}_{m}, \end{equation}
   	 \begin{equation}\label{for3}
   	 \mathbb{V}_{m+1}^{bT}\star\mathcal{A}\star\mathbb{V}^b_{m}=\mathbb{H}_{m+1}, \end{equation}
   	 \begin{equation}\label{for4}
   	 \mathbb{V}_{m}^{bT}\star \mathbb{V}^b_{m}=\mathscr{I}_{ms } \end{equation}
   	 where $\mathscr{I}_{ms }\in \mathbb{R}^{ms\times   ms \times n_{3} },$ denote the identity tensor.
   	 \end{proposition}
  \medskip
  \noindent Notice that in  the case  $n_3=1$, Algorithm \ref{TBA} reduces to the well known block   Arnoldi process. 
  \begin{proof}
  	     The proof come directly from steps of Algorithm \ref{TBA} and Proposition \ref{propoblock}. 
  \end{proof}
   
    \subsection{Tubal-Block-Arnoldi method for solving large Sylvester tensor equations (TBAS)}
     \noindent This subsection discusses the computation of an approximate solution of the tensor equations  : 
     \begin{equation}\label{syslintens}
     {\mathcal M} (\mathscr{X}) = \mathscr{C},
     \end{equation}
     where  ${\mathcal M}$ is a linear operator  that could be described as 
     \begin{equation}\label{eq21}
     {\mathcal M} (\mathscr{X}) = \mathscr{A} \star\mathscr{X}-\mathscr{X} \star\mathscr{B}
     \end{equation}
     where  $\mathscr{A}\in \mathbb{R}^{n \times n \times n_{3}},\mathscr{B}\in \mathbb{R}^{q \times q \times n_{3}}$ and $\mathscr{X},\mathscr{C}\in \mathbb{R}^{n \times q \times n_{3}}$ respectively. $\mathscr{A},\mathscr{B}$ and $\mathscr{C}$ are  given, $\mathscr{X}$ the unknown tensor to  be determined.\\
     \noindent Let $\mathscr{  {X}}_{0}\in \mathbb{R}^{n\times s\times p}$ be an arbitrary initial guess with   the corresponding  residual
     $\mathscr{R}_0=\mathscr{C}-\mathcal{M}( \mathscr{X}_0) $.    The purpose of tensor Tubal Block Arnoldi method for solving large Sylvester tensor equations  (TBAS) method is to find and approximate solution  $\mathscr{X}_{m}$ of the exact solution $\mathscr{X}^*$ of    \eqref{syslintens}  such that 
     \begin{equation}
     \label{TBAS}
     \mathscr{X}_{m}-\mathscr{X}_{0}\in  \mathcal {K}_m^{Block}(\mathcal{M},\mathscr{R}_0), 
     \end{equation} 
     with the classical orthogonality  property 
     \begin{equation}
     \label{TBASortho} 
      \mathscr{R}_{m}  =   \mathscr{C}-\mathcal{M}(\mathscr{X}_m)  \perp  \mathcal {K}_m^{Block}(\mathcal{M},\mathscr{R}_0).
     \end{equation} 
      \medskip
      \noindent  Using the  fact that $\mathcal {K}_m^{Block}(\mathcal{M},\mathscr{R}_0)=\mathcal {K}_m^{Block}(\mathscr{A},\mathscr{R}_0)$ (which not difficult to prove), the relations  (\ref{TBAS}) and (\ref{TBASortho} ) can be expressed as 
       \begin{equation}
       \label{TBAS1}
       \mathscr{X}_{m}-\mathscr{X}_{0}\in  \mathcal {K}_m^{Block}(\mathscr{A},\mathscr{R}_0), 
       \end{equation} 
       with the classical orthogonality  property 
       \begin{equation}
       \label{TBASortho1} 
       \mathscr{R}_{m}  =   \mathscr{C}-\mathcal{M}(\mathscr{X}_m)  \perp  \mathcal {K}_m^{Block}(\mathscr{A},\mathscr{R}_0).
       \end{equation} 
       \medskip
       \noindent Since   $\mathbb{V}^b_m$ (defined earlier) form an  orthonormal basis of  of the Krylov subspace  $\mathcal {K}_m^{Block}(\mathscr {A}, \mathscr{R}_0)$. The relations  (\ref{TBAS1}) and (\ref{TBASortho1}) can be written as
       \begin{equation}
       \label{TBAS12}
       \mathscr{X}_{m}=\mathscr{X}_{0}+\mathbb{V}^b_m\star \mathscr{Y}_{m}, \quad \text{ with }  \mathscr{Y}_{m}\in \mathbb{R}^{ms\times s\times n_3}, 
       \end{equation} 
       and 
       \begin{equation}
       \label{TBASortho12} 
       \mathbb{V}^{bT}_m\star \mathscr{R}_{m}  = 0.%  \mathscr{C}-\mathcal{M}(\mathscr{X}_m)  \perp  \mathcal {K}_m^{Block}(\mathscr{A},\mathscr{R}_0).
       \end{equation} 
      \noindent Using the fact that $\mathscr{R}_{m}  =    \mathscr{C}-\mathcal{M}(\mathscr{X}_m)=\mathscr{C}-\mathscr{A}\star \mathscr{X}_m+\mathscr{X}_m\star \mathscr{B}$,   the relations (\ref{TBAS12}) and (\ref{TBASortho12}), we get  the low
      dimensional equation 
      \begin{equation}
      \label{TBslowrank} 
      \mathbb{H}_m\star \mathscr{Y}_m-\mathscr{Y}_m\star \mathscr{B}=\mathscr{C}_1
      \end{equation} 
      where $\mathscr{C}_1=\mathbb{V}^{bT}_m\star \mathscr{R}_{0} $.\\
      \noindent The tensor equation (\ref{TBslowrank})   will be solved
      by using the  \textit{t-Bartels-Stewart} method.
    %  \medskip
      \noindent
      \begin{proposition}\label{reidue}
      	 At  step m, the  residual  norm of  $\mathscr{R}_{m}  =    \mathscr{C}-\mathcal{M}(\mathscr{X}_m)=\mathscr{C}-\mathscr{A}\star \mathscr{X}_m+\mathscr{X}_m\star \mathscr{B}$  with
      	 $\mathscr{X}_{m}=\mathscr{X}_{0}+\mathbb{V}^b_m\star \mathscr{Y}_{m}$  can be expressed  as follows :
      	 \begin{equation}
      	      ||\mathscr{R}_{m}||_F=  ||\mathscr{H}_{m+1,m}\star\mathbb{E}_{m} \star \mathscr{Y}_{m}  ||_F
      	 \end{equation}
      	 where $\mathbb{E}_{m}=\left[\mathscr{O} _{ssn_3},\ldots,\mathscr{O} _{ssn_3},\mathscr{I} _{ssn_3}\right]\in \mathbb{R}^{s \times   ms \times n_{3}} $ and $\mathscr{Y}_{m}$ solution of equation  (\ref{TBslowrank}).
      \end{proposition}
      
 \begin{proof}
 	  \noindent At  step m, the  residual $\mathscr{R}_{m}  =    \mathscr{C}-\mathcal{M}(\mathscr{X}_m)=\mathscr{C}-\mathscr{A}\star \mathscr{X}_m+\mathscr{X}_m\star \mathscr{B}$  with
 	 $\mathscr{X}_{m}=\mathscr{X}_{0}+\mathbb{V}^b_m\star \mathscr{Y}_{m}$  can be written  as :
 	 \begin{equation*}
 	 \mathscr{R}_m=\mathscr{R}_0-\mathscr{A}\star\mathscr{V}_m\star \mathscr{Y}_m +\mathscr{V}_m\star\mathscr{Y}_m\star \mathscr{B}.
 	 \end{equation*}
 	   Then, using  the fact  that  $ \mathcal{A}\star\mathbb{V}_{m}^b= \mathbb{V}^b_{m}\star \mathbb{H}_{m}+ \mathcal{V}^b_{m+1}\star(\mathscr{H}_{m+1,m}\star \mathbb{E}_{m} )$, the relation (\ref{TBslowrank} )  and the fact  that the  tensor $\mathscr{V}_{m+1}^b$ is orthogonal, we get     
 	 \begin{equation*}
 	 ||\mathscr{R}_{m}||_F=  ||\mathscr{H}_{m+1,m}\star\mathbb{E}_{m} \star \mathscr{Y}_{m}  ||_F
 	 \end{equation*}
 	\end{proof}

      \medskip
      \noindent To save memory and CPU-time requirements, the Tubal Block Arnoldi  method for solving large Sylvester tensor equation
       (TBAS) will be used in a restarted mode. This means
      that we have to restart the algorithm every m inner iterations, where m is a fixed integer. The restarted Tubal Block Arnoldi
      algorithm for solving (\ref{syslintens}), denoted by TBAS(m), is summarized as follows:\\	
      \begin{algorithm}[h!]
      	\caption{The  Tubal-Block-Arnoldi for solving large Sylvester tensor equation   TBAS(m)}  \label{Tbasm}
      	\begin{enumerate}
      		\item 	{\bf Input.} $\mathscr{A} \in \mathbb{R}^{n\times n \times n_3}$, $\mathscr{V}$, $\mathscr{B} \in \mathbb{R}^{q\times q \times n_3}$, $\mathscr{X}_{0},\mathscr{C}\in \mathbb{R}^{n\times q \times n_3}$, the maximum number of iteration  $\text{Iter}_{\text{max}} $  an integer $m$ and a tolerance  $tol$.
      		\item  Compute $\mathscr{R}_0=\mathscr{C}-\mathcal{M}( \mathscr{X}_0) $.
      		\item For $k=1,\ldots,\text{Iter}_{\text{max}}$
      		\begin{enumerate}
      			\item  Apply Algorithm \ref{TBA} to compute  $\mathbb{V}_{m}^b$ and  $\mathbb{  {H}}_m$ .

      			\item  Apply Algorithm \ref{t-bartels} to solve  the problem : $\mathbb{H}_m\star \mathscr{Y}_m-\mathscr{Y}_m\star \mathscr{B}=\mathbb{V}^{bT}_m\star \mathscr{R}_{0} $
      			\item Compute  $\mathscr{X}_{m}=\mathscr{X}_{0}+\mathbb{V}_{m}^b \star \mathscr{Y}_m  $
      		\end{enumerate}
      		\item If $||\mathscr{R}_{m}||_F<tol$, stop 
      		\item else   $\mathscr{X}_{0}=\mathscr{X}_{m}$ and go to Step 2.
      		\item 	{\bf Output.} $\mathscr{X}_m\in \mathbb{R}^{n\times q \times n_3}$ the approximate solution of  (\ref{syslintens}).   
      		
      	\end{enumerate}
      \end{algorithm}
    \medskip 
    \noindent During the computation of the approximate solution $\mathscr{X}_m$   of  (\ref{syslintens}), we assume that  $\varGamma({\bf H})_m\cap\varGamma({\bf- B})=\O{}$,   	where the block diagonal matrix $\bf H_m $ and $-\bf B$  are defined by \eqref{dft9}, and $\varGamma({\bf H}_m), \varGamma(-{\bf B})$ denotes the set
    of eigenvalues of the matrix {\bf H}$_m$ and {\bf- B} respectively.

   	\section{Numerical experiments }
   	\medskip
   	\noindent This section performs some numerical tests for  the  Tensor Tubal-Global GMRES  and  Tensor Tubal-Global Golub Kahan   methods to solve the linear tensor problem \eqref{eq1}.  All computations were carried out using the MATLAB R2018b environment with  an  Intel(R)  Core i7-8550U CPU $@ 1.80$ GHz  and processor  8 GB. 
   	 \noindent  The   stopping criterion was   $$ ||\mathscr{R}_k||_{F}  <\epsilon,  $$ where $\epsilon=10^{-6}$ is a chosen tolerance and $\mathscr{R}_k$ the m-th residual associated to    the approximate solution  $\mathscr{X}_k$.  In all the presented tables, we reported the obtained residual norms to achieve the desired convergence, the iteration number and the corresponding cpu-time. \\
   	 \noindent  \noindent We will compare the results in our method to the results obtained by solving the equivalent problem   ${\bf C}= {\bf A}  {\bf X}-{\bf X}  {\bf B}$ where 
   	 ${\bf A}$,  ${\bf B}$ and  ${\bf C}$ are  the matrices 
   	 \begin{align*}\label{dft9m}
   	 {\bf A}&= {\rm Diag}(\widetilde {\mathscr {A}})= \left (
   	 \begin{array}{cccc}
   	 \widetilde {\mathscr {A}}^{(1)}& &&\\
   	 & \widetilde {\mathscr {A}}^{(2)}&&\\
   	 &&\ddots&\\
   	 &&&\widetilde {\mathscr {A}}^{(n_3)}\\
   	 \end{array}
   	 \right)={\rm BlockDiag}(\widetilde {\mathscr {A}}^{(1)},\ldots,\widetilde {\mathscr {A}}^{(n_3)}),\\
   	 {\bf B}&= {\rm Diag}(\widetilde {\mathscr {B}})={\rm BlockDiag}(\widetilde {\mathscr {B}}^{(1)},\ldots,\widetilde {\mathscr {B}}^{(n_3)})\\
   	 {\bf C}&= {\rm Diag}(\widetilde {\mathscr {C}})={\rm BlockDiag}(\widetilde {\mathscr {C}}^{(1)},\ldots,\widetilde {\mathscr {C}}^{(n_3)})
   	 \end{align*}
   	 and the matrices $\widetilde {\mathscr {A}}^{(i)}$,$\widetilde {\mathscr {B}}^{(i)}$ and  $\widetilde {\mathscr {C}}^{(i)}$ 's are the frontal slices of the tensor $\widetilde {\mathscr {A}}$, $\widetilde {\mathscr {B}}$ and $\widetilde {\mathscr {C}}$ respectively.\\
   	 \noindent Notice  that ${\bf A}$ is a matrix of size $(n_0\times n_0 )$,
   	 $n_0=n\times n_3$,  ${\bf B}$ is a matrix of size $(s_0\times s_0) $,
   	 $s_0=s\times n_3$ and ${\bf C}$ is a matrix of size $(n_0\times s_0) $ . \\
   	 \noindent We compared the required CPU-times (in seconds) to achieve the convergence for the  two methods:
   	 \begin{enumerate}
   	 	\item \textbf{TBAS}: The tubal block Arnoldi–Sylvester  method. 
   	 	\item  \textbf{BAS}  :  Resolution of  Sylvester matrix equation :     ${\bf C}= {\bf A}  {\bf X}-{\bf X}  {\bf B}$   by  using  the   block Arnoldi-Sylvester method introduced in \cite{elguennouni}.
   	 	\item  \textbf{tFOM}  : The tensor  full orthogonalization method.
   	 	 \item  \textbf{tGMRES}  : The tensor generalized minimal residual method. 
   	 \end{enumerate}
  
   In Table 1, we reported the obtained relative residual norms, the total number of required iterations to achieve the convergence and the corresponding cpu-times for   restarted Tubal Block Arnoldi(m).    	Consider
   the convection-diffusion equation:

   \begin{equation}\label{chaleur}
    \begin{cases} 
    -\mu\;\Delta u+c^T\nabla \;u=f  \;\; &\text{in}\;\; \left[ 0 , 1\right]^N\\
    u =0  & \text{in}\;\; \partial\varOmega
   \end{cases}
   \end{equation}

    The tensor $\mathcal{  {A}},\mathcal{  {B}}$ are  obtained  using $n_3$ frontal  slices (which are   obtained  from a standard finite difference discretization of (\ref{chaleur})). In  this  example,  the frontal  slices are  of  size $n\times n$ and  $s\times s$, given as  follows $$\mathcal{  {A}}^{(i)} =\frac{\mu}{h^2_1}{\rm tridiag} ( -1 ,2,-1  )+\frac{a_i}{ 4h_1}\left[ \begin{array}{*{20}{c}}
3&-5&1 &   &  \\
1&3&-5&\ddots &  \\
&\ddots&\ddots& \ddots&1  \\
& & 1&3 &-5  \\
&    &  &   1&3  
\end{array}  \right], $$
and  
$$\mathcal{  {B}}^{(i)} =\frac{\mu}{h^2_2}{\rm tridiag} ( -1 ,2,-1  )+\frac{b_i}{ 4h_2}\left[ \begin{array}{*{20}{c}}
	3&-5&1 &   &  \\
	1&3&-5&\ddots &  \\
	&\ddots&\ddots& \ddots&1  \\
	& & 1&3 &-5  \\
	&    &  &   1&3  
\end{array}  \right]. $$For $i=1,\ldots,n_3$, we set $a_i=i, b_i= n_3+i, h_1=\frac{1}{n+1}$ and  $h_2=\frac{1}{s+1}$.  \\
\noindent In this example, the
right-hand side tensor $\mathcal {C}$ is  constructed    using the Matlab  command  $\mathcal{  {C}}={\rm rand}(n,s,n_3).$\\
\begin{table}[!htbp]
	\caption{Results for Example 2. $\epsilon=10^{-6}$,  and  $n_3=2$}\label{tab241}
	\begin{center}
		\begin{tabular}{lcccc}
			\hline $n \backslash s\backslash k$ & Method  & $\#$ its. & $ ||\mathscr{R}_k||_{F} $    & cpu-time in seconds \\
			\hline   & \textbf{TBAS}   &11 & $5.55\times 10^{-7}$  &  \textbf{3.03}\\
	        $1000  \backslash 3  \backslash  10$&\textbf{BAS}\cite{elguennouni}   &65 &$7.18\times 10^{-7}$&7.22  \\
	         & \textbf{tFOM}  &  66 & $7.57\times 10^{-7}$  & \textbf{57.69}\\
	          & \textbf{tGMRES} & 39 & $9.50\times 10^{-7}$  & \textbf{ 25.34}\\
			% & & & 6 &3 & $ 7.19\times 10^{-14}$ & 0.63  \\  
		 	\hline     & \textbf{TBAS}  &  12 & $6.86\times 10^{-8}$  & \textbf{12.02}\\
		$2000  \backslash  3  \backslash  6$&	\textbf{BAS}\cite{elguennouni}  & 69 &$  3.25\times 10^{-7}$& 17.61 \\
		 & \textbf{tFOM}  &  69 & $6.91\times 10^{-7}$  & \textbf{300.16}\\
		  & \textbf{tGMRES}  &  41 & $7.38\times 10^{-7}$  & \textbf{ 108.05}\\
	 
			 \hline  
		\end{tabular}
	\end{center}
\end{table}
\medskip \noindent Table \ref{tab241} reports on the obtained relative residual norms and the corresponding cpu-times to obtain the desired convergence. As can be seen from this table, the restarted  Tensor Tubal block Arnoldi method (TBAS)(m) gives good results with a small cpu-times.

 	\section*{Conlusion}   
 		In this paper, we introduced  new tensor krylov subspace    methods  for solving large Sylvester   tensor equations. The proposed method uses the well-known T-product for tensors and tensor subspaces.    We developed some new tensor products and the related algebraic properties. These new products will lead us to develop   third-order the tensor FOM (tFOM), tensor GMRES (tGMRES),  tubal Block Arnoldi and   the tensor tubal Block Arnoldi method to solve large Sylvester tensor equation . We give some properties related to these method. The numerical experiments show that the restarted Tensor Tubal block Arnoldi method TBAS(m) gives the best results with a small cpu-times.
 		\vspace{0.3cm}
   	\\ \textbf{Author contribution } All three authors have contributed in the same way.
  	\vspace{0.3cm}
  	\\  \textbf{Funding } None.
   	\vspace{0.3cm}
\\ \textbf{Data availability } Data sharing not applicable to this article as no data sets were generated or analyzed
  during the current study. We just generalized existing programs.\\
     	\vspace{0.3cm}
 \section*{Declarations}$ $\\
  \textbf{Ethical approval } Not applicable. External Review board.	\vspace{0.3cm}
  \\
 \textbf{ Competing interests } The authors declare no competing interests.	\vspace{0.3cm}
 \\
\textbf{  Conflict of interest } The authors declare that they have no conflict of interest.   			\vspace{0.3cm}
\\

   \end{document}